\documentclass{birkjour}
\usepackage{amsmath, amsthm, amssymb, mathabx}
\begin{document}
\newtheorem{theo}{Theorem}[section]
\newtheorem{defin}[theo]{Definition}
\newtheorem{rem}[theo]{Remark}
\newtheorem{lem}[theo]{Lemma}
\newtheorem{cor}[theo]{Corollary}
\newtheorem{prop}[theo]{Proposition}
\newtheorem{exa}[theo]{Example}
\newtheorem{exas}[theo]{Examples}
\numberwithin{equation}{section}
%%%%%%%%%%%%%%%%%%%%%%%%%%%%%%%%%%%
%%%%%%%%%%%%%%%%%%%%%%%%%%%%%%%%%%%%
%
\title[Strongly indefinite variational structure]{On differential systems with strongly indefinite variational structure}
\author[C. J. Batkam]{Cyril Jo\"el Batkam}
\address{D\'epartement de Math\'ematiques,
\\
Universit\'e de Sherbrooke,
\\
Sherbrooke, Qu\'ebec, \\J1K 2R1, CANADA.}
\email{cyril.joel.batkam@usherbrooke.ca}
\author[F. Colin]{Fabrice Colin}
\address{Department of Mathematics and Computer Science,\\
Laurentian University,\\
Ramsey Lake Road,\\ Sudbury, Ontario,\\
P3E 2C6, CANADA.}
\email{fcolin@cs.laurentian.ca}
\author[T. Kaczynski]{Tomasz Kaczynski}
\address{D\'epartement de Math\'ematiques,
\\
Universit\'e de Sherbrooke,
Sherbrooke, Qu\'ebec,\\ J1K 2R1, CANADA.}
\email{tomasz.kaczynski@usherbrooke.ca}
\begin{abstract}
We obtain multiplicity results for a class of first-order superquadratic Hamiltonian systems and a class of indefinite superquadratic elliptic systems which lead to the study of strongly indefinite functionals. There is no assumption to the effect that the nonlinear terms have to satisfy the Ambrosetti-Rabinowitz superquadratic condition. To establish the existence of solutions, a new version of the symmetric mountain pass theorem for strongly indefinite functionals is presented in this paper. This theorem is subsequently applied to deal with cases  where all the Palais-Smale sequences of the energy functional may be unbounded.
\end{abstract}
\maketitle
\begin{center}
\textit{Dedicated to Andrzej Granas}
\end{center}
%
%
%%%%%%%%%%%%%%%%%%%%%%%%%%%%%%%%%%%%%%%%%%%%%%%%%%%%
%%%%%%%%%%%%%%%%%%%%%%%%%%%%%%%%%%%%%%%%%%%%%%%%%%%%
\section{Introduction}
%%%%%%%%%%%%%%%%%%%%%%%%%%%%%%%%%%%%%%%%%%%%%%%%%%%%
%%%%%%%%%%%%%%%%%%%%%%%%%%%%%%%%%%%%%%%%%%%%%%%%%%%%
\paragraph{} This paper is concerned with the existence and the multiplicity of critical points of strongly indefinite even functionals which appear in the study of periodic solutions of Hamiltonian systems as well as solutions for some classes of elliptic systems. We recall that a functional $\Phi:X\to\mathbb{R}$ defined on a Banach space is said to be strongly indefinite if it is neither bounded from above nor from below, even on subspaces of finite codimension. Its study then gives rise to an interesting and challenging variational problem, because the usual powerful critical point theorems in \cite{A-R,B,B-W,Zou} cannot be directly applied.
\paragraph{} Critical Point Theory for strongly indefinite even functionals was studied in \cite{Ba-Cla,B-C,B-C1,B-D1,Ben,F-W,Li}. In his seminal paper \cite{Ben}, Benci introduced various index and pseudo-index theories whose definitions, however, depend on the topology of the sublevel sets of the functional. In addition, a dimension property requirement significatively restricts the class of problems for which the results can be applied. In \cite{Ba-Cla,F-W,Li}, the Galerkin type approximations were used to reduce the study of strongly indefinite functionals to a semidefinite situation where the basic idea of Lusternik-Schnirelmann theory applies. However, in order to control the critical points of the reduced functional, it is required in theses papers that the original functional satisfies a strong version of the usual compactness condition used in Critical Point Theory. In the recent papers \cite{B-C,B-C1}, the first two authors of this paper generalized the well known fountain theorem of Bartsch and Willem to strongly indefinite functionals. In contrast with \cite{Ba-Cla,F-W,Li}, no reduction method was used and the proofs were directly carried out in the infinite-dimensional setting.
\paragraph{} The first goal of this paper is to extend the results presented in \cite{B-C,B-C1} to more general situations for which the Palais-Smale sequences of the functionals may be unbounded. This will have a crucial importance for the applications, since Palais-Smale type assumptions such as the Ambrosetti-Rabinowitz superlinear condition and its variations, extensively used in literature in the study of superlinear problems, can be avoided. We would like to stress that, in contrast with the common approach in the study of symmetric strongly indefinite functionals, ours is not based on any reduction method. The main ingredient is a weak-strong topology introduced by Kryszewski and Szulkin in \cite{K-S}. Our critical point theorems are stated in Section \ref{section2}.
\paragraph{} In Section \ref{section3}, we consider the applications to the problem of finding infinitely many large energy periodic solutions for the following first-order Hamiltonian system
 \begin{equation}\label{h}\tag{HS}
   \left\{
    \begin{array}{ll}
      \partial_tu-\Delta_x u+V(x)u=H_v(t,x,u,v)\,\, \textnormal{ in }\mathbb{R}\times\Omega, & \hbox{} \\
      -\partial_tv-\Delta_x v+V(x)v=H_u(t,x,u,v)\,\,\textnormal{ in }\mathbb{R}\times\Omega, & \hbox{}
      \end{array}
  \right.
\end{equation}
where $\Omega\subset\mathbb{R}^N$ ($N\geq1$) is a bounded smooth domain, and $H$ is a superquadratic $C^1$-function which is $T$-periodic ($T>0$) with respect to the $t$-variable. By periodic solution, we mean a solution $z=(u,v):\mathbb{R}\times\Omega\to\mathbb{R}^{2M}$ of \eqref{h} satisfying the conditions
\begin{align*}
z(t,x) &= z(t+T,x)\quad \forall (t,x)\in\mathbb{R}\times\Omega \\
  z(t,x) &= 0\quad \forall (t,x)\in\mathbb{R}\times\partial\Omega.
\end{align*}
 Setting
\begin{equation*}
    \mathcal{J}=\left(
                  \begin{array}{cc}
                    0 & -I \\
                    I & 0\\
                  \end{array}
                \right),\qquad \mathcal{J}_0=\left(
                  \begin{array}{cc}
                    0 & I \\
                    I & 0\\
                  \end{array}
                \right),\qquad z=(u,v),
\end{equation*}
\begin{equation*}
   A=\mathcal{J}_0(-\Delta_x+V), \text{ and }\mathcal{H}=-\frac{1}{2}(Az,z)+H(t,x,z),
\end{equation*}
system \eqref{h} simply reads:
\begin{equation*}
    \mathcal{J}\partial_t z=\mathcal{H}_z(t,x,z),\quad (t,x)\in\mathbb{R}\times\Omega,
\end{equation*}
which is the form of unbounded Hamiltonian systems or infinite-dimensional Hamiltonian systems in $L^2(\Omega,\mathbb{R}^{2M})$. This kind of systems were studied under various assumptions by Brezis and Nirenberg \cite{B-N}, and by Cl\'{e}ment et al. \cite{Cle}.
\paragraph{} Usually, in the study of superquadratic Hamiltonian systems, the nonlinear terms satisfy the following condition named after Ambrosetti and Rabinowitz \cite{A-R}
\begin{equation}\label{ar1}\tag{AR}
    \exists\mu>2,\,R>0\quad;\quad 0<\mu H(t,x,z)\leq H_z(t,x,z)z\text{ for }|z|\geq R,
\end{equation}
see, for instance \cite{Bar,B-D,Bar-Szu,R} and references therein. It is well known that this condition is mainly used to obtain the boundedness of the Palais-Smale sequences of the energy functional and  without it, the problem then becomes more complicated. Moreover without \eqref{ar1}, it might happen that all the Palais-Smale sequences of the functional would be unbounded.
\paragraph{} Let
\begin{equation*}
    \widetilde{H}(t,x,u,v):=\frac{1}{2}H_z(t,x,z)\cdot z-H(t,x,z), \quad S^1:=\mathbb{R}/(T\mathbb{Z}),\,\text{ and }\,\Theta=S^1\times\Omega.
\end{equation*}
Our assumptions on the potential $V$ are the following:
\begin{enumerate}
\item [$(V_1)$] $V\in C\big(\overline{\Omega},\mathbb{R}\big)$.
\item [$(V_2)$] $0\notin\sigma(S)$, $S=-\Delta_x+V$.
\end{enumerate}
For the Hamiltonian $H$ we make the following assumptions:
\begin{enumerate}
\item [$(H_1)$] $H\in C^1\big(\Theta\times\mathbb{R}^{2M},\mathbb{R}\big)$ is $T$-periodic with respect to $t$.
\item [$(H_2)$] $H(t,x,0)=0$ and $|H_z(t,x,z)|=\circ(|z|)$ as $z\to0,$ uniformly in $(t,x)$.
\item [$(H_3)$] $H(t,x,z)>0$ for $z\neq0$ and $\frac{H(t,x,z)}{|z|^2}\to\infty$ as $|z|\to\infty,$ uniformly in $(t,x)$.
\item [$(H_4)$] $\exists R>0$ such that
\begin{itemize}
\item[$(1)$] $\widetilde{H}(t,x,z)\geq a_1|z|^2$ if $|z|\geq R$,
\item [$(2)$] $\Big(\frac{|H_z(t,x,z)|}{|z|}\Big)^\sigma\leq a_2\widetilde{H}(t,x,z)$ if $|z|\geq R$,
\end{itemize}
where $a_1,a_2>0$, $\sigma>1$ if $N=1$ and $\sigma>\frac{N}{2}+1$ if $N\geq2$.
\end{enumerate}
 Under these assumptions, Mao et al. \cite{Mao-al} obtained a nontrivial periodic solution for \eqref{h} by using a local linking theorem. We will prove here that when the Hamiltonian is even with respect to $z$, the number of periodic solutions is in fact infinite. This has already been observed by Bartsch and Ding in \cite{B-D} where they studied \eqref{h} with $\Omega$ bounded and with $\Omega=\mathbb{R}^N$. However, they used condition \eqref{ar1} to verify the boundedness of the Palais-Smale sequences of the energy functional associated to \eqref{h}, which was crucial for their argument. It is well known that \eqref{ar1} implies that $H(t,x,z)\geq c|z|^\mu$ for $|z|$ large, therefore it is stronger than $(H_3)$. %Moreover, one does not know whether the solutions they obtained there are large energy solutions.
We would like to emphasize that in our situation, determining whether the Palais-Smale sequences of the energy functional are bounded or not can be a very difficult, if not impossible, task.\\
 Our result reads as follows:
\begin{theo}\label{resultat}
Assume that $(V_1)$, $(V_2)$, and $(H_1)$-$(H_4)$ are satisfied. If in addition $H$ is even in $z$, then \eqref{h} has infinitely many pairs $\pm z_k$ of $T$-periodic solutions such that $\|z_k\|_\infty\to\infty$ as $k\to\infty$.
\end{theo}
\paragraph{} Finally, as another application of our abstract result, we consider in Section \ref{section4} the following elliptic system of Hamiltonian type
\begin{equation}\label{s}\tag{ES}
    \left\{
      \begin{array}{ll}
        -\Delta u=g(x,v)\,\, \textnormal{in }\Omega, & \hbox{} \\
        -\Delta v=f(x,u)\,\,\textnormal{in }\Omega, & \hbox{} \\
        u=v=0\textnormal{ on }\partial\Omega, & \hbox{}
      \end{array}
    \right.
\end{equation}
where $\Omega$ is a bounded smooth domain in $\mathbb{R}^N$, $N\geq3$. The solutions of this problem describe steady states of some reaction-diffusion systems that derive from several applications, coming from mathematical biology or from the modelization of chemical reactions.
\paragraph{} We study this problem under the following assumptions:
\begin{enumerate}
  \item [$(E_1)$]$f,g\in\mathcal{C}(\Omega\times\mathbb{R})$ and there is a constant $C>0$ such that
\begin{equation*}
    |f(x,u)|\leq C(1+|u|^{p-1}) \textnormal{ and } |g(x,u)|\leq C(1+|u|^{q-1})\quad\forall (x,u),
\end{equation*}
\begin{equation*}
 \text{where } p,q>2\text{ satisfy }  \frac{1}{p}+\frac{1}{q}>1-\frac{2}{N}.
\end{equation*}
Furthermore, in case $N\geq5$ we impose
\begin{equation*}
    \frac{1}{p}>\frac{1}{2}-\frac{2}{N}\quad\textnormal{and}\quad \frac{1}{q}>\frac{1}{2}-\frac{2}{N}.
\end{equation*}
  \item [$(E_2)$] $f(x,u)=\circ(u)$ and $g(x,u)=\circ(u)$ as $u\to0$, uniformly in $x$. \\
  \item [$(E_3)$] $F(x,u)/u^2\to\infty$ and $G(x,u)/u^2\to\infty$, uniformly in $x$, as $|u|\to\infty$.\\
  \item [$(E_4)$]  $u\mapsto f(x,u)/|u|$ and $u\mapsto g(x,u)/|u|$ are increasing in $(-\infty,0)\cup(0,+\infty)$.\\
  \item [$(E_5)$] $f(x,-u)=-f(x,u)$ and $g(x,-u)=-g(x,u)$ for all $(x,u)$.
\end{enumerate}
Before stating our result for this problem, we recall the following definition.
\begin{defin}
We say that $(u,v)$ is a strong solution of \eqref{s} if $u\in W^{2,p/(p-1)}(\Omega)\cap W_0^{1,p/(p-1)}(\Omega)$, $v\in W^{2,q/(q-1)}(\Omega)\cap W_0^{1,q/(q-1)}(\Omega)$ and $(u,v)$ satisfies
\begin{equation*}
    \left\{
      \begin{array}{ll}
        -\Delta u=g(x,v)\,\, \textnormal{ a.e. in }\Omega, & \hbox{} \\
        -\Delta v=f(x,u)\,\,\textnormal{ a.e. in }\Omega. & \hbox{}
        \end{array}
    \right.
\end{equation*}
\end{defin}
We will also prove the following
\begin{theo}\label{mainresult2}
Under assumptions $(E_1)$-$(E_5)$, \eqref{s} has infinitely many pairs of strong solutions $\pm(u_k,v_k)$ such that $\|(u_k,v_k)\|\to\infty$ as $k\to\infty$, where $\|\cdot\|$ represents the norm in the space
\begin{equation*}
    W^{2,p/(p-1)}(\Omega)\cap W_0^{1,p/(p-1)}(\Omega)\times W^{2,q/(q-1)}(\Omega)\cap W_0^{1,q/(q-1)}(\Omega).
\end{equation*}
\end{theo}
System \eqref{s} has been already studied from the variational point of view by many authors. Hulshof and van der Vorst \cite{H-V}, and de Figueiredo and Felmer \cite{Figue-Fel} obtained positive solutions by requiring among others that the Hamiltonian $H(u,v)=F(u)+G(v)$ satisfies \eqref{ar1}. In \cite{F-W}, Felmer and Wang considered the case where the full superquadratic range is not reached, and by using a variation of \eqref{ar1}, they obtained infinitely many solutions when $H$ is even in $(u,v)$ by means of the Galerkin approximation. Recently, Szulkin and Weth \cite{S-W} improved the above results by reducing the energy functional to the Nehari-Pankov manifold. To circumvent the difficulty that the Nehari-Pankov manifold is not necessary of class $C^1$, they required the maps in assumption $(E_4)$ to be strictly increasing. In \cite{Bat}, the first author of this paper showed that it is sufficient to assume that these maps are only increasing. He applied the generalized fountain theorem obtained in \cite{B-C2} to a family of perturbed functionals. However, working with a family of modified functionals could make  things unnecessary complicated. We considerably simplify this approach in the present paper.\\
\paragraph{} Throughout the paper, we denote by $|\cdot|_r$ the norm of the Lebesgue space $L^r$.

%%%%%%%%%%%%%%%%%%%%%%%%%%%%%%%%%%%%%%%%%%%%%%%%%%%%
%%%%%%%%%%%%%%%%%%%%%%%%%%%%%%%%%%%%%%%%%%%%%%%%%%%%
\section{Critical point theorems for strongly indefinite functionals}\label{section2}
%%%%%%%%%%%%%%%%%%%%%%%%%%%%%%%%%%%%%%%%%%%%%%%%%%%%
%%%%%%%%%%%%%%%%%%%%%%%%%%%%%%%%%%%%%%%%%%%%%%%%%%%%
Let $X$ be a separable Hilbert space which admits an orthogonal decomposition $X=X^-\oplus X^+$, where $X^-$ is closed and $X^+=(X^-)^\perp$. We denote by $\big<,\big>$ the inner product of $X$. \\
Let $(a_j)_{j\geq0}$ be an orthonormal basis of $X^-$. We define on $X$ a new norm by setting
\begin{equation*}
    \vvvert u\vvvert:=\max\Big(\sum\limits_{j=0}^{\infty}\frac{1}{2^{j+1}}|\big<P^-u,a_j\big>|,\|P^+u\|\Big),\,\, u\in X,
\end{equation*}
where $P^\pm:X\to X^\pm$ are orthogonal projections, and we denote by $\tau$ the topology generated by this norm. This topology was introduced by Kryszewski and Szulkin in \cite{K-S} and is related to the topology on $X$ which is strong on $X^+$ and weak on bounded sets of $X^-$. More precisely, if $(u_n)$ is a bounded sequence in $X$ then
\begin{equation*}
    u_n
\stackrel{\tau}{\rightarrow}u  \Longleftrightarrow P^-u_n \rightharpoonup P^-u \,\ \text{and} \,\ P^+u_n \rightarrow P^+u.
\end{equation*}
Now we recall some standard notations in Critical Point Theory:\\
Let $\Phi:X\to\mathbb{R}$ and $a,b\in\mathbb{R}$. We denote by
\begin{equation*}
    \Phi^b:=\big\{u\in X\,\ ; \,\ \Phi(u)\leq a\big\},\quad \Phi_a:=\big\{u\in X\,;\, \Phi(u)\geq a\big\},\text{ and }\Phi_a^b:=\Phi_a\cap \Phi^b.
\end{equation*}
We say that a functional $\Phi\in C^1(X,\mathbb{R})$ satisfies the Palais-Smale condition at the level $c\in \mathbb{R}$ if the following holds:
\begin{itemize}
  \item [$(PS)_c$] Any sequence $(u_n)\subset X$ such that $\Phi(u_n)\to c$ and $\Phi'(u_n)\to0$ \big(Palais-Smale sequence at the level $c$\big) admits a convergent subsequence.
\end{itemize}
This is the famous compactness condition in Critical Point Theory introduced by Palais and Smale \cite{P-S}. Since we would like to consider some situations where this condition does not hold, we will use its following weaker version due to Cerami \cite{Ce}:
 \begin{itemize}
  \item [$(Ce)_c$] Any sequence $(u_n)\subset X$ such that $\Phi(u_n)\to c$ and $(1+\|u_n\|)\Phi'(u_n)\to0$ \big(Cerami sequence at the level $c$\big) admits a convergent subsequence.
\end{itemize}
If $(Ce)_c$ holds we say that the functional $\Phi$ satisfies the Cerami condition at the level $c$ \big($(Ce)_c$-condition for short\big).
\paragraph{} We now consider the class of $C^1$ functionals $\Phi:X\to\mathbb{R}$ which satisfying:
\begin{itemize}
  \item [$(KS)$] $\Phi(u)=\frac{1}{2}\|P^+u\|^2-\frac{1}{2}\|P^-u\|^2-\Psi(u)$, where $\Psi\in\mathcal{C}^1(X,\mathbb{R})$ is bounded below, weakly sequentially lower semicontinuous, and $\Psi'$ is weakly sequentially continuous.
\end{itemize}
\paragraph{} The following deformation lemma will play a key role in the proof of our abstract results.
\begin{lem}[Deformation lemma]\label{dl}
Assume that $\Phi$ satisfies $(KS)$. Let $d\geq b$ and $\varepsilon>0$ such that
\begin{equation}\label{hypo}
    \forall u\in \Phi^{-1}([b-2\varepsilon,d+2\varepsilon]), \, \big(1+\|u\|\big)\|\Phi'(u)\|\geq 8\varepsilon.
\end{equation}
Then there exists $\eta \in\mathcal{C}([0,1]\times\Phi^{d+2\epsilon},X)$ such that:
\begin{itemize}
  \item [(i)] $\eta(0,u)=u$ for all $u\in\Phi^{d+2\varepsilon},$
  \item [(ii)] $\eta(1,\Phi^{d+\varepsilon})\subset\Phi^{b-\varepsilon},$
  \item [(iii)] $\Phi(\eta(\cdot,u))$ is non increasing, $\forall u\in \Phi^{d+2\varepsilon}$,
  \item [(iv)] each point $(t,u)\in [0,1]\times\Phi^{d+2\varepsilon}$ has a $\tau$-neighborhood $N_{(t,u)}$ such that $\big\{v-\eta(s,v)\, \bigl| \, (s,v)\in N_{(t,u)}\cap([0,1]\times\Phi^{d+2\varepsilon})\bigr.\big\}$ is contained in a finite-dimensional subspace of $X$,
  \item [(v)] $\eta$ is $\tau$-continuous,
  \item [(vi)] if $\Phi$ is even then $\eta(t,\cdot)$ is odd $\forall t\in [0,1]$.
\end{itemize}
\end{lem}
\begin{proof}
 It goes back to an idea of Li and Szulkin \cite{LiSzul02}.
\paragraph{} We define the vector field
\begin{equation*}
    w(u):=2\nabla\Phi(u)/\|\Phi'(u)\|^2,\quad u\in\Phi^{-1}\big([b-2\varepsilon,d+2\varepsilon]\big).
\end{equation*}
By assumption $(KS)$ we know that $\Phi'$ is weakly sequentially semicontinuous, and this implies that the function
\begin{equation*}
    v\in\Phi^{d+2\varepsilon}_{b-2\varepsilon}\mapsto \big<\Phi'(v),w(u)\big>\in\mathbb{R}
\end{equation*}
 is $\tau$-continuous. Hence every $u\in\Phi^{d+2\varepsilon}_{b-2\varepsilon}$ has a $\tau$-neighborhood $N_u$ such that
\begin{equation}\label{cond1}
    \big<\Phi'(v),w(u)\big>>1\quad\forall v\in N_u,
\end{equation}
\begin{equation}\label{cond2}
    \|u\|<2\|v\|\quad\forall v\in N_u,
\end{equation}
where \eqref{cond2} holds because the set $\big\{z\in X\,;\,\|z\|\leq\alpha\big\}$ is $\tau$-closed for any $\alpha\geq0$.\\
Now since $(KS)$ implies that $\Phi$ is $\tau$-upper semicontinuous, the set $\widetilde{N}:=\Phi^{-1}\big(]-\infty,b-2\varepsilon[\big)$ is $\tau$-open. It follows that $\mathcal{N}:=\widetilde{N}\cup\big\{N_v\,;\,b-2\varepsilon\leq\Phi(v)\leq d+2\varepsilon\big\}$ is a $\tau$-open covering of the metric space $\big(\Phi^{d+2\varepsilon},\tau\big)$. We can then extract a $\tau$-locally finite $\tau$-open covering $\mathcal{M}:=\big\{M_i\,;\,i\in I\big\}$ of $\Phi^{c+2\epsilon}$ finer than $\mathcal{N}$.\\
Let
\begin{equation*}
    V:=\bigcup\limits_{i\in I}M_i.
\end{equation*}
For every $i\in I$ we have either $M_i\subset N_v$ for some $v$ or $M_i\subset \widetilde{N}$. In the first case we define $v_i:=w(v)$ and in the second case $v_i:=0$. Consider  a $\tau$-Lipschitz continuous partition of unity $\big\{\lambda_i\, ; \, i\in I  \big\}$ subordinated to $\mathcal{M}$ and define on $V$ the vector field
\begin{equation*}
f(u):=\sum\limits_{i\in I}\lambda_i(u)v_i.
\end{equation*}
Clearly, the vector field $f$ is locally Lipschitz continuous and $\tau$-locally Lipschitz continuous. By using \eqref{hypo} and \eqref{cond2} we see that
\begin{equation*}
    \|f(u)\|\leq\frac{1}{4\varepsilon}\big(1+2\|u\|\big)\quad\forall u\in V.
\end{equation*}
 It follows from Corollary $7.6$ in \cite{Sch} that the problem
$$\left\{
    \begin{array}{ll}
      \frac{d}{dt}\sigma(t,u)=-f(\sigma(t,u)) & \hbox{} \\
      \sigma(0,u)=u \in\Phi^{d+2\varepsilon} & \hbox{}
    \end{array}
  \right.
$$
has a unique solution $\sigma(\cdot,u)$ defined on ${\mathbb R}^+$.\\
We define $\eta:[0,1]\times \Phi^{d+2\varepsilon}$ by setting
$$\eta(t,u):=\sigma\big((d-b+2\varepsilon)t,u\big).$$
An argument similar to that in the proof of Lemma $6.8$ in \cite{W} shows that $\eta$ satisfies $(i)$-$(v)$. If $\Phi$ is even, then we replace $f(u)$ with $\frac{1}{2}\big(f(u)-f(-u)\big)$. $(vi)$ is then a consequence of the existence and the uniqueness of the solution for the above Cauchy problem.
\end{proof}
Now we can state our first critical point theorem, that extends the generalized saddle point theorem in \cite{L-S} to the case where the Palais-Smale sequences of $\Phi$ may be unbounded.
\begin{theo}[Saddle point theorem]\label{theo1}
Assume that $\Phi$ satisfies $(KS)$. If there exists $R>0$ such that
\begin{equation*}
    (A_0)\quad\quad\quad b:=\inf_{\substack{u\in X^+}}\Phi(u)>\sup_{\substack{u\in X^-\\ \|u\|=R}}\Phi(u) \quad\text{and}\quad d:=\sup_{\substack{u\in X^-\\ \|u\|\leq R}}\Phi(u)<\infty,\quad
\end{equation*}
then for some $c\in[b,d]$, there is a sequence $(u_n)\subset X$ such that
\begin{equation*}
    \Phi(u_n)\to c\quad\text{and}\quad \big(1+\|u_n\|\big)\Phi'(u_n)\to0\,\,\text{as}\,\, n\to\infty.
\end{equation*}
Moreover, if $\Phi$ satisfies the $(Ce)_\alpha$ condition for all $\alpha\in[b,d]$, then $\Phi$ has a critical value in $[b,d]$.
\end{theo}
The proof follows the lines of \cite{L-S}.
\begin{proof}
We assume by contradiction that for every $c\in[b,d]$ there is no $(Ce)_c$ sequence for $\Phi$. Then there exists $\varepsilon>0$ such that
\begin{equation*}
    u\in\Phi^{-1}\big([b-2\varepsilon,d+2\varepsilon]\big)\Rightarrow \big(1+\|u\|\big)\|\Phi'(u)\|\geq8\varepsilon.
\end{equation*}
Let
\begin{equation*}
    M:=\big\{u\in X^-\,;\, \|u\|=R\big\}.
\end{equation*}
We apply Lemma \ref{dl}, and we define $\mu:[0,1]\times M\to X$ by setting $\mu(t,u):=P^-\eta(t,u)$, where $\eta$ is given by Lemma \ref{dl}. Clearly, $(iv)$ and $(v)$ of Lemma \ref{dl} imply that $\mu$ is $\tau$-continuous and each $(t,u)\in [0,1]\times\Phi^{d+2\varepsilon}$ has a $\tau$-neighborhood $N_{(t,u)}$ such that $\big\{v-\mu(s,v)\, \bigl| \, (s,v)\in N_{(t,u)}\cap([0,1]\times\Phi^{d+2\varepsilon})\big\}$ is contained in a finite-dimensional subspace of $X$.\\
We claim that $0\notin \mu\big([0,1]\times\partial M\big)$. Indeed, if there exists $(t_0,u_0)\in[0,1]\times\partial M$ such that $\mu(t_0,u_0)=0$, then $\eta(t_0,u_0)\in X^+$ and by $(iii)$ of Lemma \ref{dl}  and assumption $(A_0)$ we have
\begin{equation*}
    \Phi(u_0)=\Phi\big(\eta(0,u_0)\big)\geq\Phi\big(\eta(t_0,u_0)\big)\geq b>\sup_{\partial M}\Phi,
\end{equation*}
which contradicts the fact that $u_0\in\partial M$.\\
$\mu$ is then an admissible homotopy \big(in the sense of Kryszewski and Szulkin \cite{K-S}\big) such that $0\notin \mu\big([0,1]\times\partial M\big)$. It follows from Theorem $2.4$-$(iii)$ of \cite{K-S} that the Kryszewski-Szulkin's degree (see \cite{K-S}) $deg_{KS}\big(\mu(t,\cdot),int(M),0\big)$ is well defined and does not depend on $t\in[0,1]$. Hence
\begin{equation*}
    deg_{KS}\big(\mu(1,\cdot),int(M),0\big)=deg_{KS}\big(\mu(0,\cdot),int(M),0\big)=deg_{KS}\big(id,int(M),0\big).
\end{equation*}
It follows from Theorem $2.4$-$(i)$ of \cite{K-S}  that there is $u\in int(M)$ such that $\mu(1,u)=0$, which implies that $\eta(1,u)\in X^+$. By the definition of $b$ we have $b\leq\Phi\big(\eta(1,u)\big)$. But $(ii)$ of Lemma \ref{dl} implies, since $M\subset\Phi^{d+\varepsilon}$, that $\Phi\big(\eta(1,u)\big)\leq b-\varepsilon$. This gives a contradiction.
\end{proof}
In order to obtain a multiplicity result we need to introduce some notations.\\
We consider an orthonormal basis $(e_j)_{j\geq0}$ of $X^+$ and we set
\begin{equation*}
    X_k^-:=X^-\oplus\big(\oplus_{j=0}^k\mathbb{R}e_j\big) \text{ and } X_k^+:=\overline{\oplus_{j=k}^\infty\mathbb{R}e_j}.
\end{equation*}
\begin{theo}[Fountain theorem]\label{theo2}
Assume that $\Phi$ satisfies $(KS)$, that $\Phi$ is even, and that there exist $\rho_k>r_k>0$ such that:
\begin{equation*}
(A_1) \quad\quad a_k:=\sup_{\substack{u\in X_k^-\\ \|u\|=\rho_k}}\Phi(u)<\min\big(0,\inf_{\substack{u\in X\\ \|u\|\leq r_k}}\Phi(u)\big), \quad d_k:=\sup_{\substack{u\in X_k^-\\ \|u\|\leq\rho_k}}\Phi(u)<\infty.\quad\quad\quad
\end{equation*}
\begin{equation*}
 (A_2) \quad\quad\quad\quad\quad\quad\quad\quad\quad\quad   b_k:=\inf_{\substack{u\in X_k^+\\ \|u\|=r_k}}\Phi(u)\to\infty,\quad k\to\infty.\quad\quad\quad\quad\quad\quad\quad
\end{equation*}
Then, there exist a sequence $(u_k^n)_n\subset X$ and a number $c_k\geq b_k$ such that
\begin{equation*}
    \Phi(u_k^n)\to c_k\quad\text{and}\quad \big(1+\|u_k^n\|\big)\Phi'(u_k^n)\to0\,\,\text{as}\,\, n\to\infty.
\end{equation*}
Moreover, if $\Phi$ satisfies the $(Ce)_c$ condition for any $c>0$, then $\Phi$ has an unbounded sequence of critical values.
\end{theo}
Theorem \ref{theo2} generalizes Theorem $12$ in \cite{B-C} where the sequence $(u_k^n)$ was a Palais-Smale sequence.
\begin{proof}
Let $B_k:=\big\{u\in X_k^-\,;\, \|u\|\leq\rho_k\big\}$ and let $\Gamma_k$ be the set of maps $\gamma:B_k\rightarrow X$ such that:
\begin{itemize}
  \item [(a)] $\gamma$ is odd and $\tau$-continuous,
  \item [(b)] each $u\in B_k$ has a $\tau$-neighborhood $\mathcal{N}_u$ in $X^-_k$ such that $(id-\gamma)\bigl(\mathcal{N}_u\cap B_k\bigr)$ is contained in a finite dimensional subspace of $X$,
  \item [(c)] $\Phi(\gamma(u))\leq\Phi(u)$, $\forall u\in B_k$.
\end{itemize}
Let $\gamma\in \Gamma_k$ and let $U=\{u\in B_k \, ; \, \|\gamma(u)\|<r_k\bigr.\}$. It obviously follows from $(a)$ and $(b)$ that $\gamma$ is continuous (in the norm topology). If $u\in B_k$ is such that $\|u\|=\rho_k$, then by $(c)$ we have $\Phi(\gamma(u))<\Phi(u)$, so it follows from $(A_1)$ that $u\notin U$. Hence $U$ is a symmetric $\tau-$open neighborhood of $0$ in $X_k^-$.  It is clear that $B_k$ is $\tau$-closed, so we deduce from the $\tau$-continuity of $\gamma$ that $\overline{U}$ is also $\tau$-closed. \\
  Consider the map 
  $$P_{k-1}\gamma:\overline{U}\rightarrow X^-_{k-1},$$
  where $P_{k-1}:X\to X^-_{k-1}$ is the orthogonal projection.
\begin{itemize}
\item $P_{k-1}\gamma$ is $\tau$-continuous. 
  \item Let $u\in U$. From $(b)$ $u$ has a $\tau$-neighborhood $N_u$ such that $(id-\gamma)(N_u\cap U)\subset W$, where $W$ is a finite-dimensional subspace of $X$. Let $v\in N_u\cap U\subset X^-_{k}=X^-_{k-1}\oplus {\mathbb R} e_k$, then $(id-P_{k-1}\gamma)(v)=P_{k-1}(v-\gamma(v))+\lambda e_k\in W+{\mathbb R} e_k$ which is finite-dimensional.
\end{itemize} 
It follows from Theorem $5$ in \cite{B-C} that there exists $u_0\in \partial U$ such that $P_{k-1}\gamma(u_0)=0$. This implies that
\begin{equation}\label{inter}
\gamma\big(B_k\big)\cap \big\{u\in X^+_k\,;\, \|u\|=r_k\big\}\neq\emptyset.
\end{equation}
Now $(A_1)$ and $(A_2)$ imply that there is $k_0$ big enough such that $b_k>a_k$ for every $k\geq k_0$. We fix $k\geq k_0$ and we define
\begin{equation*}
     c_{k} :=  \inf_{\gamma \in \Gamma_{k}} \sup_{u\in B_{k}} \Phi
\bigl( \gamma(u)  \bigr).
\end{equation*}
The intersection property \eqref{inter} implies that $c_k\geq b_k$.
\paragraph{} We assume by contradiction that there is no Cerami sequence of $\Phi$ at the level $c_k$. Then there exists $\varepsilon\in(0,\frac{c_k-a_k}{2})$ such that
\begin{equation*}
     u\in \Phi^{-1}([c_k-2\varepsilon,c_k+2\varepsilon])\quad\Rightarrow\quad \big(1+\|u\|\big)\|\Phi'(u)\|\geq 8\varepsilon,
\end{equation*}
where $\gamma\in \Gamma_k$ is such that
\begin{equation}\label{e.6}
    \sup_{B_{k}} \Phi \circ \gamma \leq c_{k}+\varepsilon.
\end{equation}
We apply Lemma \ref{dl} with $b=d=c_k$ and we define on $B_k$ the map
\begin{equation*}
    \beta(u):=\eta(1,\gamma(u)).
\end{equation*}
It follows from $(iii),(iv),(v),(vi)$ of Lemma \ref{dl} that $\beta\in\Gamma_k$.\\
Now by using $\eqref{e.6}$ and $(ii)$ of Lemma \ref{dl} we obtain
\begin{equation*}
  { \displaystyle \sup_{\substack{u \in X^-_k \\ \|u\| = \rho_k}} \Phi(\beta(u)) = \displaystyle \sup_{\substack{u \in X^-_k \\ \|u\| = \rho_k}} \Phi(\eta(1,\gamma(u))) \leq c_k-\varepsilon}
\end{equation*}
giving a contradiction with the definition of $c_k$. It then follows that $\Phi$ has a Cerami sequence at level $c_k$.
\end{proof}
Using the same argument as in the proof of Theorem $6$ in \cite{B-C1}, we can prove the following dual version of Theorem \ref{theo2}.
\begin{theo}[Dual fountain theorem]
Let $\Phi(u)=\frac{1}{2}\|u^+\|^2-\frac{1}{2}\|u^-\|^2+\Psi(u)$, where $\Psi\in\mathcal{C}^1(X,\mathbb{R})$ is even, bounded below, and weakly sequentially lower semicontinuous, and $\Psi'$ is weakly sequentially continuous.\\ If $\forall k\geq k_0$, $\exists\rho_k>r_k>0$ such that
\begin{itemize}
  \item [$(B_1)$]\,\,\,\  {$a^k \, := \, \displaystyle \inf_{\substack{u \in X^+_{k}\\ \|u\| = \rho_k}} \Phi(u) >\max\big(0, \sup_{\substack{u\in X\\ \|u\|\leq r_k}}\Phi(u)\big)$},\quad {$b^k \, := \, \displaystyle \sup_{\substack{u \in X^{-}_k\\ \|u \| = r_k}} \Phi(u) <0$,}
  \item [$(B_2)$]\,\,\,\ {$d^k \, := \, \displaystyle \inf_{\substack{u \in X^+_k\\ \|u \| \leq \rho_k}} \Phi(u) \rightarrow 0, \, k \rightarrow \infty$,}
\end{itemize}
Then there exist $c^k\in[d^{k},0[$ and a sequence $(u_k^n)\subset X$ such that
\begin{equation*}
\Phi(u_k^n)\to c_k\quad\text{and}\quad \big(1+\|u_n\|\big)\Phi'(u_k^n)\to0,\quad \text{as }n\to\infty.
\end{equation*}
Moreover, if $\Phi$ satisfies the $(Ce)_c$ condition for any $c\in [d^{k_0},0[$, then $\Phi$ has a sequence $(u_k)$ of critical points such that $\Phi(u_k)\to0^-$ as $k\to\infty$.
\end{theo}

%%%%%%%%%%%%%%%%%%%%%%%%%%%%%%%%%%%%%%%%%%%%%%%%%%%%
%%%%%%%%%%%%%%%%%%%%%%%%%%%%%%%%%%%%%%%%%%%%%%%%%%%%
\section{Periodic solutions of superquadratic Hamiltonian systems}\label{section3}
%%%%%%%%%%%%%%%%%%%%%%%%%%%%%%%%%%%%%%%%%%%%%%%%%%%%
%%%%%%%%%%%%%%%%%%%%%%%%%%%%%%%%%%%%%%%%%%%%%%%%%%%%
In this section, we apply our multiplicity result (Theorem \ref{theo2}) to the Hamiltonian system
 \begin{equation*}\tag{HS}
   \left\{
    \begin{array}{ll}
      \partial_tu-\Delta_x u+V(x)u=H_v(t,x,u,v)\,\, \textnormal{ in }\mathbb{R}\times\Omega, & \hbox{} \\
      -\partial_tv-\Delta_x v+V(x)v=H_u(t,x,u,v)\,\,\textnormal{ in }\mathbb{R}\times\Omega. & \hbox{}
      \end{array}
  \right.
\end{equation*}
We assume here that assumptions $(V_1)$, $(V_2)$, and $(H_1)$-$(H_4)$ are satisfied, and that $H$ is even in $z$.
\paragraph{} By $(V_1)$ the operator $S=-\Delta_x+V$ acting on $L^2(\Omega,\mathbb{R}^{2M})$ is self-adjoint with domain $D(S)=H^2(\Omega,\mathbb{R}^{2M})\cap H^1_0(\Omega,\mathbb{R}^{2M})$.\\
Let $L:=\mathcal{J}\partial_t+A$. Observing that the operators $A$ and $\mathcal{J}A$ acting on $L^2(\Omega,\mathbb{R}^{2M})$ with domains $D(A)=D(\mathcal{J}A)=H^2(\Omega,\mathbb{R}^{2M})\cap H^1_0(\Omega,\mathbb{R}^{2M})$ are self-adjoint, then the operator $L$ acting on $L^2\big(S^1,L^2(\Omega,\mathbb{R}^{2M})\big)$ is also self-adjoint. We know by \cite{B-D} that
\begin{equation*}
L^2\big(S^1,L^2(\Omega,\mathbb{R}^{2M})\big)\cong L^2\big(S^1\times\Omega,\mathbb{R}^{2M}\big),\quad 0\notin\sigma(A)\cup\sigma(\mathcal{J}A),\quad 0\notin\sigma(L),
\end{equation*}
so there is an orthogonal decomposition
\begin{equation*}
L^2(S^1\times\Omega,\mathbb{R}^{2M})=E^+\oplus E^-,
\end{equation*}
such that $L$ is negative on $E^-$, and positive on $E^+$.\\
The space $X=D(|L|^{\frac{1}{2}})$ is a Hilbert space with the inner product
\begin{equation*}
\big<w,z\big>=\big<|L|^{\frac{1}{2}}w,|L|^{\frac{1}{2}}z\big>_{L^2}
\end{equation*}
and norm $\|z\|^2=\big<z,z\big>$.\\
Moreover, we have an orthogonal decomposition
\begin{equation*}
    X=X^-\oplus X^+\text{ with } X^\pm=X\cap E^\pm.
\end{equation*}
We have the following useful lemma:
\begin{lem}[\cite{B-D}, Lemma $4.6$]\label{compa}
The embedding of $X$ in $L^r\big(S^1\times\Omega,\mathbb{R}^{2M}\big)$ is compact for $r\geq2$ if $N=1$, and for $r\in[2,2(N+2)/2)$ if $N\geq2$.
\end{lem}
\paragraph{} Let's define the functional $\Phi:X\to\mathbb{R}$,
\begin{equation}\label{phii}
\Phi(z)=\frac{1}{2}\|z^+\|^2-\frac{1}{2}\|z^-\|^2-\Psi(z), \,\text{ where }\,\Psi(z)=\int_\Theta H(t,x,z),
\end{equation}
which is such that its critical points are also weak solutions of \eqref{h}.
\begin{lem}\label{regular}\qquad
\begin{enumerate}
\item $\Phi\in\mathcal{C}^1(X,\mathbb{R})$.
\item $\Psi$ is bounded below and is weakly sequentially lower semicontinuous.
\item $\Psi'$ is weakly sequentially continuous.
\end{enumerate}
\end{lem}
\begin{proof}
\begin{enumerate}
  \item By $(H_4)$-$(2)$ we have  for $|z|\geq R$ and for $(t,x)\in \Theta$
\begin{align*}
 |H_z(t,x,z)|^\sigma &\leq a_2\widetilde{H}(t,x,z)|z|^\sigma\\
 & \leq a_2\Big(\frac{1}{2}H_z(t,x,z)z-H(t,x,z)\Big)|z|^\sigma\\
&\leq \frac{a_2}{2}|H_z(t,x,z)||z|^{\sigma+1}.
\end{align*}
Hence
\begin{equation*}
|H_z(t,x,z)|\leq a_3|z|^{\frac{\sigma+1}{\sigma-1}} \quad \forall |z|\geq R,\quad\forall (t,x)\in\Theta,
\end{equation*}
where $a_3>0$ is constant. \\
On the other hand, since $H_z$ is continuous, we have $|H_z(t,x,z)|\leq a_4$, for $|z|\leq R$ and $(t,x)\in \Theta$. We deduce that
\begin{equation}\label{un}
|H_z(t,x,z)|\leq a_4+a_3|z|^{\frac{\sigma+1}{\sigma-1}} \quad \forall z\in \mathbb{R}^{2M},\quad\forall (t,x)\in\Theta.
\end{equation}
Let $\varepsilon>0$. $(H_2)$ implies that there is $\delta>0$ such that
\begin{equation*}
|H_z(t,x,z)|\leq\varepsilon|z|\quad\text{pour tout }|z|<\delta.
\end{equation*}
If $|z|\geq\delta$, then we have in view of \eqref{un}
\begin{equation*}
|H_z(t,x,z)|\leq a_4+a_3|z|^{\frac{\sigma+1}{\sigma-1}}\leq a_4\frac{|z|^p}{\delta^p}+a_3|z|^p=\big(\frac{a_4}{\delta^p}+a_3\big)|z|^p, 
\end{equation*}
with  $p=\frac{\sigma+1}{\sigma-1}$. Hence
\begin{equation}\label{deux}
\forall\varepsilon>0,\,\,\exists C(\varepsilon)>0\,\,;\,\, |H_z(t,x,z)|\leq \varepsilon|z|+C(\varepsilon)|z|^{p}\quad\forall z\in\mathbb{R}^{2M}.
\end{equation}
Now one can easily verify that
\begin{equation*}
\left\{
  \begin{array}{ll}
    \sigma>1\text{ if }N=1 & \hbox{} \\
    \sigma>1+\frac{N}{2}\text{ if }N\geq2 & \hbox{}
  \end{array}
\right.
 \Longrightarrow
\left\{
  \begin{array}{ll}
    p+1>2\text{ si }N=1 & \hbox{} \\
    2<p+1<2(N+2)/N)\text{ if }N\geq2. & \hbox{}
  \end{array}
\right.
\end{equation*}
It then follows from Lemma \ref{compa} that $\Psi$ is well defined on $X$. We can now use a standard argument to show that $\Phi$ is of class $C^1$ on $X$ and
\begin{equation}\label{phiprim}
\big<\Phi'(z),w\big>=\big<z^+-z^-,w\big>-\int_\Theta wH_z(t,x,z),\quad\forall z,w\in X.
\end{equation}
  \item Since $H\geq0$, the functional $\Psi$ is bounded below. Let $z_n\rightharpoonup z$ and $c\in\mathbb{R}$ such that $\Psi(z_n)\leq c$. By Lemma \ref{compa} we have $z_n\to z$ in $L^2$ and, up to a subsequence, $z_n(t,x)\to z(x,t)$ a.e. in $\Theta$. Using Fatou's lemma we obtain $\Psi(z)\leq c$, which shows that $\psi$ is weakly lower semicontinuous.
  \item Let $z_n\rightharpoonup z$. By Lemma \ref{compa} we have $z_n\to z$ in $L^{p+1}$ \big(where $p=\frac{\sigma+1}{\sigma-1}$\big) which implies, using \eqref{deux} and Theorem $A.2$ in \cite{W}, that $|H(t,x,z_n)-H(t,x,z)|_{\frac{p+1}{p}}\to0$. It is then easy to deduce, using Fatou's lemma, that $\Psi'$ is weakly sequentially continuous.
\end{enumerate}
\end{proof}
We now prove that the functional $\Phi$ satisfies the Cerami condition.
\begin{lem}\label{cerami}
Let $(z_n)\subset X$ be such that
\begin{equation*}
d=\sup\Phi(z_n)<\infty\quad\text{and}\quad \big(1+\|z_n\|\big)\Phi'(z_n)\to0.
\end{equation*}
Then $(z_n)$ admits a convergent subsequence.
\end{lem}
\begin{proof}
Let us first show that the sequence $(z_n)$ is bounded.\\
Arguing by contradiction we suppose that $(z_n)$ is  unbounded. Then, up to a subsequence, we have $\|z_n\|\to\infty$.\\
Observe that
\begin{align*}
\big<\Phi'(z_n),z_n^+-z_n^-\big>&=\|z_n^+\|^2+\|z_n^-\|^2-\int_\Theta (z_n^+-z_n^-)\cdot H_z(t,x,z_n)\\
&=\|z_n\|^2\Big(1-\int_\Theta\frac{ (z_n^+-z_n^-)\cdot H_z(t,x,z_n)}{\|z_n\|^2}\Big).
\end{align*}
Since
\begin{equation*}
|\big<\Phi'(z_n),z_n^+-z_n^-\big>|\leq\|\Phi'(z_n)\|\|z_n^+-z_n^-\|\leq2\|z_n\|\|\Phi'(z_n)\|\to0,
\end{equation*}
one must have
\begin{equation}\label{etoile}
\int_\Theta\frac{ (z_n^+-z_n^-)\cdot H_z(t,x,z_n)}{\|z_n\|^2}\to1.
\end{equation}
 Set $w_n=z_n/\|z_n\|$.\\
Assumption $(H_4)$-$(1)$ implies that $\widetilde{H}(t,x,z)\geq a_1|z|^2-c_1$ for all $(t,x,z)$. We deduce that for $n$ big enough one has
\begin{align}
1+d\geq\Phi(z_n)-\frac{1}{2}\big<\Phi'(z_n),z_n\big>&=\int_\Theta \widetilde{H}(t,x,z_n)\label{trois}\\
&\geq a_2|z_n|_2^2-c_1|\Theta|,\label{quatre}
\end{align}
where $|\Theta|$ denotes the Lebesgue  measure of $\Theta$. \eqref{quatre} shows that $(|z_n|_2)$ is bounded. Hence
\begin{equation}\label{cinq}
|w_n|_2=\frac{|z_n|_2}{\|z_n\|}\to0.
\end{equation}
On the other hand we have
\begin{align*}
\Big|\int_\Theta\frac{(z_n^+-z_n^-)\cdot H_z(t,x,z_n)}{\|z_n\|^2}\Big|&\leq \int_\Theta\frac{|z_n^+-z_n^-|}{\|z_n\|}\frac{|H_z(t,x,z_n|}{\|z_n\|}\\
&=\int_\Theta|w_n^+-w_n^-||w_n|\frac{|H_z(t,x,z_n|}{|z_n|}\\
\leq\Big(\int_\Theta &\big(\frac{|H_z(t,x,z_n|}{|z_n|}\big)^\sigma\Big)^{1/\sigma}\Big(\int_\Theta \big(|w_n^+-w_n^-||w_n|\big)^{\sigma'}\Big)^{1/\sigma'},
\end{align*}
with $\frac{1}{\sigma}+\frac{1}{\sigma'}=1.$
Since
\begin{equation*}
    |w_n^+-w_n^-||w_n|=|w_n^+-w_n^-||w_n^++w_n^-|\leq(|w_n^+|+|w_n^-|)^2\leq 2(|w_n^+|^2+|w_n^-|^2),
\end{equation*}
we have
\begin{equation*}
\int_\Theta \big(|w_n^+-w_n^-||w_n|\big)^{\sigma'}\leq c_2\Big(|w_n^+|_{2\sigma'}^{2\sigma'}+|w_n^-|_{2\sigma'}^{2\sigma'}\big).
\end{equation*}
By using the fact that $2\sigma'=p+1\in[2,2(N+2)/N]$, we have
\begin{equation*}
|w_n^\pm|_{2\sigma'}\leq|w_n^\pm|_2^\alpha|w_n^\pm|_{2(N+2)/N}^{1-\alpha}\quad \big(\text{interpolation inequality}\big),
\end{equation*}
where $\alpha\in[0,1]$ is such that $\frac{\alpha}{2}+\frac{1-\alpha}{2(N+2)/N}=1$. Since $X$ continuously embeds in $L^{2(N+2)/N}$, we have $|w_n^\pm|_{2(N+2)/N}\leq c_3\|w_n^\pm\|\leq c_3$ (because $\|w_n^\pm\|\leq\|w_n\|=1$).\\
Now, since the decomposition $X=X^-\oplus X^+$ is also orthogonal with respect to the $L^2$-norm, we have $|w_n^\pm|_2\leq|w_n|_2$. Hence $|w_n^\pm|_{2\sigma'}\leq c_3|w_n|_2^\alpha$, which implies
\begin{equation}\label{six}
\Big|\int_\Theta\frac{(z_n^+-z_n^-)\cdot H_z(t,x,z_n)}{\|z_n\|^2}\Big|\leq c_4\Big(\int_\Theta \Big(\frac{|H_z(t,x,z_n|}{|z_n|}\Big)^\sigma\Big)^{1/\sigma}|w_n|_2^{2\alpha}.
\end{equation}
Let $\varepsilon>0$. By $(H_2)$ there exists $\delta>0$ such that $|H_z(t,x,z)|\leq\varepsilon|z|$ for $|z|<\delta$. Then
\begin{align*}
\int_\Theta \Big(\frac{|H_z(t,x,z_n|}{|z_n|}\Big)^\sigma&=\int_{\Theta\cap\{|z_n|<\delta\}} \Big(\frac{|H_z|}{|z_n|}\Big)^\sigma+\int_{\Theta\cap\{\delta\leq|z_n|\leq R\}} \Big(\frac{|H_z|}{|z_n|}\Big)^\sigma\\
& \qquad \qquad\qquad+\int_{\Theta\cap\{|z_n|\geq R\}} \Big(\frac{|H_z|}{|z_n|}\Big)^\sigma\\
\leq |\Theta|\varepsilon^\sigma+&\sup_{\substack{\Theta\cap\{\delta\leq|z_n|\leq R\}}}\Big(\frac{|H_z|}{|z_n|}\Big)^\sigma|\Theta|+a_2\int_{\Theta\cap\{|z_n|\geq R\}}\widetilde{H}(t,x,z_n),
\end{align*}
where we have used $(H_4)$-$(2)$ for the last term in the RHS of the equality. We deduce from \eqref{trois} that
$\Big(\int_{\Theta\cap\{|z_n|\geq R\}}\widetilde{H}(t,x,z_n)\Big)$ is bounded. Hence $\Big(\int_\Theta \big(\frac{|H_z(t,x,z_n|}{|z_n|}\big)^\sigma\Big)$ is bounded. It then follows from \eqref{six} and \eqref{cinq} that
\begin{equation*}
\int_\Theta\frac{(z_n^+-z_n^-)\cdot H_z(t,x,z_n)}{\|z_n\|^2}\to0,
\end{equation*}
which contradicts \eqref{etoile}. This contradiction leads to the conclusion that the sequence $(z_n)$ is bounded.
\paragraph{} We may now suppose that $z_n\rightharpoonup z$ in $X$. Since
\begin{equation*}
\|z_n^\pm-z^\pm\|^2=\pm\big<\Phi'(z_n)-\Phi'(z),z_n^\pm-z^\pm\big>\pm\int_\Theta(z_n^\pm-z^\pm)\cdot H_z(t,x,z_n),
\end{equation*}
a standard argument using Lemma \ref{compa} shows that $z_n\to z$.
\end{proof}

\begin{proof}[Proof of Theorem \ref{resultat}]
Let $(e_i)$ be an orthonormal basis of $X^+$. We recall that
\begin{equation*}
    X^-_k=X^-\oplus\big(\oplus_{j=0}^k\mathbb{R}e_j\big)\text{ and }X^+_k=\overline{\oplus_{j=k}^\infty\mathbb{R}e_j}.
\end{equation*}

 Let $z\in X^-_k$. Using $(H_3)$ we have
\begin{equation*}
\forall \delta>0\quad\exists c_\delta>0\,\,;\,\ H(t,x,z) \geq\delta|z|^2-c_\delta.
\end{equation*}
Then
\begin{equation*}
\Phi(z)=\frac{1}{2}\|z^+\|^2-\frac{1}{2}\|z^-\|^2-\int_\Theta H(t,x,z)\leq \frac{1}{2}\|z^+\|^2-\frac{1}{2}\|z^-\|^2-\delta|z|_2^2+c_\delta|\Theta|.
\end{equation*}
Since the decomposition $X=X^-\oplus X^+$ is orthogonal with respect to the $L^2$-norm, we have $|z^+|_2\leq|z|_2$. Since all norms are equivalent on $\oplus_{j=0}^k\mathbb{R}e_j$, there is $c_1>0$ such that $c_1\|z^+\|^2\leq|z|_2^2$. Hence
\begin{equation*}
\Phi(z)\leq\big(\frac{1}{2}-c_1\delta\big)\|z^+\|^2-\frac{1}{2}\|z^-\|^2+c_1|\Theta|.
\end{equation*}
Let us choose $\delta$ such that $c_1\delta\geq1$. Then we have
\begin{equation}\label{sept}
\Phi(z)\leq-\frac{1}{2}\|z\|^2+c_\delta|\Theta|\to-\infty\text{ as }\|z\|\to\infty.
\end{equation}
\paragraph{} Let $z\in X^+_k$. Then for every $\varepsilon>0$, we have
\begin{align*}
\Phi(z)&=\frac{1}{2}\|z\|^2-\int_\Theta H(t,x,z)\\
&\geq \frac{1}{2}\|z\|^2-\frac{\varepsilon}{2}|z|_2^2-\frac{c(\varepsilon)}{p+1}|z|_{p+1}^{p+1}\,\,(\text{by }\eqref{deux})\\
&\geq\frac{1}{2}\|z\|^2-c_2\varepsilon|z|_2^2-\frac{c(\varepsilon)}{p+1}|z|_{p+1}^{p+1}\,\,(\text{because } X\hookrightarrow L^2)\\
&\geq \big(\frac{1}{2}-c_2\varepsilon\big)\|z\|^2-\frac{c(\varepsilon)}{p+1}\beta_k^{p+1}\|z\|^{p+1},
\end{align*}
where $\beta_k=\sup_{\substack{w\in X^+_k\\\|w\|=1}}|w|_{p+1}$. By choosing $\varepsilon\leq\frac{1}{2c_2}$ we obtain
\begin{equation*}
\Phi(z)\geq\frac{1}{2}\Big(\frac{1}{2}\|u\|^2-c_3\beta_k^{p+1}\|z\|^{p+1}\Big).
\end{equation*}
Hence, since $\beta_k\to0$ as $k\to\infty$ (see \cite{W}, Lemma $3.8$), we have
\begin{equation}\label{huit}
\|z\|=r_k:=\big(c_3(p+1)\beta_k^{p+1}\big)^{1/(1-p)}\Rightarrow\Phi(z)\geq\frac{1}{2}\big(\frac{1}{2}-\frac{1}{p+1})r_k^2\to\infty,\,k\to\infty.
\end{equation}
\paragraph{} \eqref{sept} and \eqref{huit} show that the conditions $(A_1)$ and $(A_2)$ of Theorem \ref{theo2} are satisfied. In view of Lemmas \ref{regular} and \ref{cerami}, we can apply Theorem \ref{theo2} and get the result.
\end{proof}

%%%%%%%%%%%%%%%%%%%%%%%%%%%%%%%%%%%%%%%%%%%%%%%%%%%%
%%%%%%%%%%%%%%%%%%%%%%%%%%%%%%%%%%%%%%%%%%%%%%%%%%%%
\section{Strong solutions of superquadratic elliptic systems}\label{section4}
%%%%%%%%%%%%%%%%%%%%%%%%%%%%%%%%%%%%%%%%%%%%%%%%%%%%
%%%%%%%%%%%%%%%%%%%%%%%%%%%%%%%%%%%%%%%%%%%%%%%%%%%%
In this section, we consider the problem
\begin{equation*}\tag{ES}
    \left\{
      \begin{array}{ll}
        -\Delta u=g(x,v)\,\, \textnormal{in }\Omega, & \hbox{} \\
        -\Delta v=f(x,u)\,\,\textnormal{in }\Omega, & \hbox{} \\
        u=v=0\textnormal{ on }\partial\Omega, & \hbox{}
      \end{array}
    \right.
\end{equation*}
and we assume in the sequel that the assumptions $(E_1)$-$(E_5)$ are satisfied.
\paragraph{} Consider the Laplacian as the operator
\begin{equation*}
    -\Delta:H^2(\Omega)\cap H_0^1(\Omega)\subset L^2(\Omega)\to L^2(\Omega),
\end{equation*}
and let $(\varphi_j)_{j\geq1}$ a corresponding system of orthogonal and $L^2(\Omega)$-normalized eigenfunctions, with eigenvalues $(\lambda_j)_{j\geq1}$. Then, writing
\begin{equation*}
    u=\sum\limits_{j=1}^\infty a_j\varphi_j,\quad \textnormal{with } a_j=\int_\Omega u\varphi_j ,
\end{equation*}
we set, for $0\leq s\leq2$
\begin{equation*}
    E^s:=\Big\{u\in L^2(\Omega)\, \big|\,\sum\limits_{j=1}^\infty\lambda_j^s|a_j|^2<\infty \Big\}
\end{equation*}
and
\begin{equation*}
    A^s(u):=\sum\limits_{j=1}^\infty\lambda_j^{s/2}a_j\varphi_j,\quad \forall u\in D(A^s)=E^s.
\end{equation*}
One can verify easily that $A^s$ is an isomorphism onto $L^2(\Omega)$. We denote $A^{-s}:=(A^s)^{-1}$. It is well known (see \cite{Lions-Ma}) that the space $E^s$ is a fractional Sobolev space with the inner product
\begin{equation*}
    \big<u,v\big>_s=\int_\Omega A^suA^sv.
\end{equation*}
\begin{lem}[\cite{Per}]\label{compactness}
$E^s$ embeds continuously in $L^r(\Omega)$ for $s>0$ and $r\geq1$ satisfying $\frac{1}{r}\geq\frac{1}{2}-\frac{s}{N}.$ Moreover, the embedding is compact when the inequality is strict.
\end{lem}
By assumption $(E_1)$, there exist $s,t>0$ such that $s+t=2$ and
\begin{equation}\label{pq}
    \frac{1}{p}>\frac{1}{2}-\frac{s}{N}\quad \textnormal{and}\quad \frac{1}{q}>\frac{1}{2}-\frac{t}{N}.
\end{equation}
We consider the functional
\begin{equation*}
    \Phi(u,v):=\int_\Omega A^suA^tv-\Psi(u,v),\quad (u,v)\in E^s\times E^t.
\end{equation*}
where
\begin{equation*}
    \Psi(u,v)=\int_\Omega\Big(F(x,u)+G(x,v)\Big).
\end{equation*}
It follows from Lemma \ref{compactness} that the inclusions $E^s\hookrightarrow L^p(\Omega)$ and  $E^t\hookrightarrow L^q(\Omega)$ are continuous. This fact in conjunction with the estimate
\begin{equation*}
   \Big|\int_\Omega A^suA^tv\Big|\leq  |A^su|_2|A^tu|_2=\|u\|_s\|v\|_t,
\end{equation*}
imply that the preceding functional $\Phi$ is well defined on $E:=E^s\times E^t$.\\
Now a standard argument shows that if assumption $(E_1)$ holds, then the functional $\Phi$ is of class $\mathcal{C}^1$ on $E$ and its critical points are weak solutions of \eqref{s}.\\
We recall that $(u,v)\in E^s\times E^t$ is a weak solution of \eqref{s} if
\begin{multline*}
    \int_\Omega \big(A^suA^tk+A^shA^tv\big)-\int_\Omega\Big(hf(x,u)+kg(x,v)\Big)=0,\quad\forall(h,k)\in E^s\times E^t.
\end{multline*}
 The following regularity result is due to de Figueiredo and Felmer \cite{Figue-Fel}.
\begin{lem}
If $(u,v)\in E^s\times E^t$ is a weak solution of \eqref{s}, then $(u,v)$ is a strong solution of \eqref{s}.
\end{lem}
\paragraph{} We endow $E=E^s\times E^t$ with the inner product
\begin{equation*}
    \big<(u,v),(\phi,\varphi)\big>_{s\times t}=\big<u,\phi\big>_s+\big<v,\varphi\big>_t,\quad (u,v),(\phi,\varphi)\in E,
\end{equation*}
and the associated norm $\|(u,v)\|_{s\times t}^2=\big<(u,v),(u,v)\big>_{s\times t}$.
\paragraph{} In the following we may assume without loss of generality that $s\geq t$.\\
 One can easily verify that $E$ has the orthogonal decomposition \big(with respect to $\big<\cdot,\cdot\big>_{s\times t}$\big) $E=E^+\oplus E^-$, where
\begin{equation}\label{}
    E^+:=\Big\{(u,A^{s-t}u)\,|\,u\in E^s\Big\}\quad \textnormal{and}\quad E^-:=\Big\{(u,-A^{s-t}u)\,|\,u\in E^s\Big\}.
\end{equation}
If we denote by $P^\pm:E\to E^\pm$ the orthogonal projections, then a direct calculation yields
\begin{equation}\label{}
    P^\pm(u,v)=\frac{1}{2}\big(u\pm A^{t-s}v,v\pm A^{s-t}u\big),\quad \forall (u,v)\in E,
\end{equation}
and
\begin{equation}\label{phi}
    \Phi(u,v)=\frac{1}{2}\|P^+(u,v)\|_{s\times t}^2-\frac{1}{2}\|P^-(u,v)\|_{s\times t}^2-\int_\Omega\Big(F(x,u)+G(x,v)\Big).
\end{equation}
\paragraph{} Let us recall the following technical result, which will play a crucial role in the verification of the Cerami condition.
\begin{prop}[\cite{Liu}, Lemma $3.2$]\label{liu}
Let $x\in\mathbb{R}^N$, $u,v,s\in\mathbb{R}$ such that $s\geq-1$ and $su+v\neq0$. If $(E_2)$-$(E_4)$ are satisfied then
\begin{align*}
     f(x,u)\Big[s(\frac{s}{2}u+(1+s)v)\Big]&+F(x,u)-F(x,u+v)\leq0,\\
g(x,u)\Big[s(\frac{s}{2}u+(1+s)v)\Big]&+G(x,u)-G(x,u+v)\leq0.
\end{align*}
\end{prop}
\begin{lem}\label{conditioncerami}
Let $(z_n=(u_n,v_n))\subset E$ and $c>0$ such that
\begin{equation*}
\Phi(z_n)\to c\quad\text{and}\quad \big(1+\|z_n\|_{s\times t}\big)\Phi'(z_n)\to0.
\end{equation*}
Then $(z_n)$ has a convergent subsequence.
\end{lem}
\begin{proof}
Let us first show that $(z_n)$ is bounded.\\
On the contrary if $(z_n)$ is  unbounded, then we can assume that $\|z_n\|_{s\times t}\to\infty$. Set $w_n=z_n/\|z_n\|_{s\times t}$. Then, up to a subsequence, we have $w_n\rightharpoonup w$ in $E$.\\
Since $F,G\geq0$, we have for $n$ big enough
\begin{equation*}
c-1\leq\Phi(z_n)\leq\frac{1}{2}\|z_n^+\|_{s\times t}^2-\frac{1}{2}\|z_n^-\|_{s\times t}^2=\frac{1}{2}\|z_n\|_{s\times t}^2\big(\|w_n^+\|_{s\times t}^2-\|w_n^-\|_{s\times t}^2\big).
\end{equation*}
Recall that $\|w_n^+\|_{s\times t}^2+\|w_n^-\|_{s\times t}^2=1$. Hence $\|w_n^+\|_{s\times t}^2\geq\frac{1}{2}+\frac{c-1}{\|z_n\|_{s\times t}^2}$.
But since $\|z_n\|_{s\times t}^2\to\infty$, we have for $n$ big enough $\frac{|c-1|}{\|z_n\|_{s\times t}^2}\leq\frac{1}{6}$. It follows that $\|w_n^+\|_{s\times t}^2\geq\frac{1}{3}$.\\
Now fix $r>0$ and set $s_n=\frac{r}{\|z_n\|_{s\times t}}-1$ and $y_n=-\frac{r}{\|z_n\|_{s\times t}}z_n^-=-rw_n^-$. Then $s_n\geq-1$, $y_n\in E^-$ and the sequences $(s_n)$ and $(y_n)$ are bounded.\\
By setting $y_n=(a_n,b_n)$, where $a_n\in E^s$ and $b_n\in E^t$, we have
\begin{align*}
\big<\Phi'(z_n),s_n(\frac{s_n}{2}+1)z_n&+(1+s_n)y_n\big>=\big <z_n^+-z_n^-,s_n(\frac{s_n}{2}+1)z_n+(1+s_n)y_n\big>\\
&-\int_\Omega\big(s_n(\frac{s_n}{2}+1)u_n+(1+s_n)a_n\big)f(x,u_n)\\
&-\int_\Omega\big(s_n(\frac{s_n}{2}+1)v_n+(1+s_n)b_n\big)g(x,v_n).
\end{align*}
Since $(s_n)$ and $(y_n)$ are bounded, we obtain
\begin{align*}
\big|\big<\Phi'(z_n),s_n(\frac{s_n}{2}+1)z_n+(1+s_n)y_n\big>\big|&\leq\|\Phi'(z_n)\|\| s_n(\frac{s_n}{2}+1)z_n+(1+s_n)y_n\|\\
&\leq C_1\|\Phi'(z_n)\|\|z_n\|_{s\times t}+C_2\|\Phi'(z_n)\|\to0. %\quad \big(\text{car }(1+\|z_n\|)\Phi'(z_n)\to0\big).
\end{align*}
We deduce that for $n$ big enough
\begin{align}
\nonumber \big <z_n^+-z_n^-,s_n(\frac{s_n}{2}+1)z_n&+(1+s_n)y_n\big>\leq 1\\
\nonumber &+\int_\Omega\big(s_n(\frac{s_n}{2}+1)u_n+(1+s_n)a_n\big)f(x,u_n)\\
&+\int_\Omega\big(s_n(\frac{s_n}{2}+1)v_n+(1+s_n)b_n\big)g(x,v_n).\label{etoille}
\end{align}
Noting that $rw_n^+=z_n+(s_nz_n+y_n)$ and
\begin{align*}
\big<z_n^+-z_n^-,s_n(\frac{s_n}{2}+1)z_n+(1+s_n)y_n\big>=-\frac{1}{2}\|z_n^+\|_{s\times t}^2+\frac{1}{2}\|z_n^-\|_{s\times t}^2\\
+\frac{1}{2}r^2\|w_n^+\|_{s\times t}^2+\frac{1}{2}r^2\|w_n^-\|_{s\times t}^2,
\end{align*}
we have
\begin{align*}
\Phi(rw_n^+)-\Phi(z_n)&=-\frac{1}{2}\|w_n^-\|_{s\times t}^2+\big<z_n^+-z_n^-,s_n(\frac{s_n}{2}+1)z_n+(1+s_n)y_n\big>\\
&+\int_\Omega\Big(F(x,u_n)-F(x,(1+s_n)u_n+a_n)\Big)\\
&+\int_\Omega\Big(G(x,u_n)-G(x,(1+s_n)u_n+a_n)\Big).
\end{align*}
We deduce by using \eqref{etoille} that for $n$ sufficiently large
\begin{align*}
\Phi(rw_n^+)-\Phi(z_n)&\leq 1+\int_\Omega\big(s_n(\frac{s_n}{2}+1)u_n+(1+s_n)a_n\big)f(x,u_n)\\
&\qquad+\int_\Omega\Big[F(x,u_n)-F\big(x,(1+s_n)u_n+a_n\big)\Big]\\
&\qquad+ \int_\Omega\big(s_n(\frac{s_n}{2}+1)v_n+(1+s_n)b_n\big)g(x,u_n)\\
&\qquad+\int_\Omega\Big[G(x,u_n)-G\big(x,(1+s_n)v_n+b_n\big)\Big].
\end{align*}
We conclude by using Proposition \ref{liu} that for $n$ big enough
\begin{equation*}
\Phi(rw_n^+)-1\leq\Phi(z_n).
\end{equation*}
If $w^+=0$, then $w_n^+\to0$ in $L^p(\Omega)\times L^q(\Omega)$ by Lemma \ref{compactness}. By $(E_1)$ and Theorem A.2 in \cite{W} we have 
\begin{equation*}
\int_\Omega F\big(x,(1+s_n)u_n+a_n\big)\to0 \text{ and } \int_\Omega G\big(x,(1+s_n)v_n+b_n\big)\to0.
\end{equation*}
Since $\Phi(z_n)\to c$, we then obtain for $n$ sufficiently large
\begin{align*}
c+1\geq\Phi(z_n)&\geq\Phi(rw_n^+)-1\\
&= \frac{1}{2}r^2\|w_n\|_{s\times t}^2-\int_\Omega F\big(x,(1+s_n)u_n+a_n\big)\\
&\qquad-\int_\Omega G\big(x,(1+s_n)v_n+b_n\big)-1\\
&\geq \frac{1}{2}r^2\|w_n\|_{s\times t}^2-2\\
&\geq\frac{r^2}{6}-2\quad\big(\text{because }\|w_n^+\|_{s\times t}^2\geq\frac{1}{3}\big).
\end{align*}
This leads to a contradiction if we take $r$ arbitrary large. Hence $w^+\neq0$ and then $w\neq0$.\\
Setting $w_n=(w_n^1,w_n^2)\in E^s\times E^t$, we have
\begin{align*}
\frac{\Phi(z_n)}{\|z_n\|_{s\times t}^2}&=\frac{1}{2}\|w_n^+\|_{s\times t}^2-\frac{1}{2}\|w_n^-\|_{s\times t}^2-\int_\Omega\frac{F(x,u_n)}{\|z_n\|_{s\times t}^2}-\int_\Omega\frac{G(x,u_n)}{\|z_n\|_{s\times t}^2}\\
&=\frac{1}{2}\|w_n^+\|_{s\times t}^2-\frac{1}{2}\|w_n^-\|_{s\times t}^2-\int_{w_n^1\neq0}\frac{F(x,w^1_n\|z_n\|_{s\times t})}{\big|w_n^1\|z_n\|_{s\times t}\big|^2}\\
&\quad-\int_{w_n^2\neq0}\frac{G(x,w^2_n\|z_n\|_{s\times t})}{\big|w_n^1\|z_n\|_{s\times t}\big|^2}\\
&\leq C-\int_{w_n^1\neq0}\frac{F(x,w^1_n\|z_n\|_{s\times t})}{\big|w_n^1\|z_n\|_{s\times t}\big|^2}-\int_{w_n^2\neq0}\frac{G(x,w^2_n\|z_n\|_{s\times t})}{\big|w_n^1\|z_n\|_{s\times t}\big|^2}.
\end{align*}
Since $(\Phi(z_n))$ is bounded and $\|z_n\|_{s\times t}\to\infty$, the LHS converges to $0$. By applying Fatou's lemma and using $(E_3)$, we see that the RHS goes to $-\infty$, which is a contradiction. Consequently the sequence $(z_n)$ is bounded.
\paragraph{} Now a standard argument based on Lemma \ref{compactness} shows that the bounded sequence $(z_n)$ has a convergent subsequence.
\end{proof}
\begin{proof}[Proof of Theorem \ref{mainresult2}]
Let $(a_j)_{j\geq0}$ be an orthonormal basis of $E^s$. We define an orthonormal basis $(e_j)_{j\geq0}$ of $E^+$ by setting
\begin{equation*}
    e_j:=\frac{1}{\sqrt{2}}\big(a_j,A^{s-t}a_j\big).
\end{equation*}
 Let
\begin{equation*}
    E^-_k=E^-\oplus\big(\oplus_{j=0}^k\mathbb{R}e_j\big)\text{ and }E^+_k=\overline{\oplus_{j=k}^\infty\mathbb{R}e_j}.
\end{equation*}
$(E_3)$ implies that for every $\delta>0$ there is $C_\delta>0$ such that
\begin{equation}\label{superquadratic}
    F(x,u)\geq \delta|u|^2-C_\delta\quad\text{and}\quad G(x,u)\geq \delta|u|^2-C_\delta,\quad \forall (x,u).
\end{equation}
\paragraph{} Let $z\in E^-_k$. Then $z=\big(u,A^{s-t}u\big)+\big(v,-A^{s-t}v\big)$, where $v\in E^s$ and $u\in E^s_k:=\oplus_{j=0}^k\mathbb{R}a_j$.
Using \eqref{superquadratic}, the fact that $E^{s-t}$ continuously embeds in $L^2(\Omega)$, and the parallelogram identity, we obtain
\begin{align*}
    \Phi(z)&=\|u\|^2_s-\|v\|^2_s-\int_\Omega\Big(F(x,u+v)+G(x,A^{s-t}(u-v))\Big)\\
&\leq\|u\|^2_s-\|v\|^2_s-2C\delta\big(|u|^2_2+|v|^2_2\big)+2C_\delta|\Omega| .
\end{align*}
Since all the norms are equivalent on the finite-dimensional subspace $E^s_k$, there is a constant $c_1>0$ such that $c_1\|u\|_s\leq|u|_2$. Hence
\begin{equation*}
    \Phi(z)\leq (1-c_2\delta)\|u\|^2_s-\|v\|^2_s-2C_\delta|\Omega|.
\end{equation*}
Choose $\delta>\frac{1}{c_2}$. Hence $\Phi(z)\to-\infty$ as $\|z\|_{s\times t}\to\infty$, and consequently assumption $(A_1)$ of Theorem \ref{theo2} is satisfied.
\paragraph{} Let $z\in E^+_k$. Then $z=(u,A^{s-t}u)$ with $u\in \overline{\oplus_{j=k}^\infty\mathbb{R}a_j}$, and
\begin{align*}
    \Phi(z)&=\frac{1}{2}\|z\|_{s\times t}^2-\int_\Omega\Big(F(x,u)+G(x,A^{s-t}u)\Big)\\
&\geq\frac{1}{2}\|z\|_{s\times t}^2-2\int_\Omega\Big(F(x,u)+G(x,A^{s-t}u)\Big).
\end{align*}
By $(E_1)$ there is a constant $C_1>0$ such that
\begin{equation*}
    |F(x,u)|\leq C_1(1+|u|^p)\quad\text{and}\quad |G(x,A^{s-t}u)|\leq C_1(1+|A^{s-t}u|^q).
\end{equation*}
Hence
\begin{equation*}
    \Phi(z)\geq\|u\|^2_s-2C_1|u|_p^p-2C_1|A^{s-t}u|^q_q-4C_1|\Omega|.
\end{equation*}
We define
\begin{equation*}
    \beta_{1,k}:=\sup_{\substack{u\in\overline{\oplus_{j=k}^\infty\mathbb{R}a_j}\\\|u\|_s=1}}|u|_p,\quad \beta_{2,k}:=\sup_{\substack{v\in\overline{\oplus_{j=k}^\infty\mathbb{R}(A^{s-t}a_j)}\\\|v\|_t=1}}|v|_q,
\end{equation*}
and $\beta_k=\max\big\{\beta_{1,k};\beta_{2,k}\big\}$.\\
Then
\begin{equation*}
    \Phi(z)\geq\|u\|^2_s-2C_1\beta_k^p\|u\|_s^p-2C_1\beta_k^q\|u\|^q_s-4C_1|\Omega|.
\end{equation*}
We may assume without loss of generality that $q\leq p$ and we set
\begin{equation*}
    r_k:=\big(C_1p\beta_k^p\big)^{\frac{1}{2-p}}.
\end{equation*}
Then for $\|z\|_{s\times t}=\sqrt{2}\|u\|_s=r_k$ we have
\begin{equation}\label{betatilde}
    \Phi(z)\geq \widetilde{b_k}:=K\beta_k^{\frac{2p}{2-p}}\Big[\Big(\frac{1}{4}-\frac{1}{p(\sqrt{2})^p}\Big)-A\beta_k^{\frac{2(q-p)}{2-p}}\Big]-4C_1|\Omega|,
\end{equation}
where $K,A>0$ are constant.\\
We know by Lemma $3.8$ in \cite{W} that $\beta_{1,k}\to0$ and $\beta_{2,k}\to0$ as $k\to\infty$. This implies that $\widetilde{b_k}\to\infty$ as $k\to\infty$, and then that assumption $(A_2)$ of Theorem \ref{theo2} is satisfied.
\paragraph{} An argument similar to the one in the proof of Lemma \ref{regular} shows that $\Psi$ is weakly sequentially lower semicontinuous and that $\Psi'$ is weakly sequentially continuous. Since we have previously shown  that the functional $\Phi$ satisfies the Cerami condition at every nonnegative critical level, we can apply Theorem \ref{theo2} and get the desired result.
\end{proof}

\section*{Acknowledgement}
This work was supported by a NSERC grant.
%%%%%%%%%%%%%%%%%%%%%%%%%%%%%%%%%%%%%%%%%%%%%%%%%%%%
%%%%%%%%%%%%%%%%%%%%%%%%%%%%%%%%%%%%%%%%%%%%%%%%%%%%
%
%\section*{References}

%
%
\end{document}